\documentclass[10pt,a4paper]{article}
\usepackage{amssymb}
\usepackage{amsmath}
\usepackage{graphicx}
\usepackage[T1]{fontenc} 
\usepackage[english]{babel}
\usepackage{amsthm}
\usepackage{tabulary}
\usepackage{booktabs}
\usepackage{mathtools}
\usepackage{tikz}
\usepackage{comment}
\usepackage{tikz-cd}
\usepackage{shuffle}
\usepackage{comment}

\usepackage{difftrees}
\tikzset{
  ncone/.pic={
	\draw (0,0)--(0,0.2);
  }
}

\tikzset{
  nctwo/.pic={
    \draw (0,0)--(0,0.2);
	\draw (0.1,0)--(0.1,0.2);
  }
}

\tikzset{
  nctwoW/.pic={
    \draw (0,0.2)--(0,0)--(0.1,0)--(0.1,0.2);
  }
}

\tikzset{
  nctwoWW/.pic={
    \draw (0,0.2)--(0,0)--(0.2,0)--(0.2,0.2);
  }
}

\tikzset{
  ncthreeWW/.pic={
    \draw (0,0.2)--(0,0)--(0.3,0)--(0.3,0.2);
	\draw (0.2,0)--(0.2,0.2);
  }
}

\tikzset{
  ncthree/.pic={
    \draw (0,0)--(0,0.2);
	\draw (0.1,0)--(0.1,0.2);
	\draw (0.2,0)--(0.2,0.2);
  }
}

\tikzset{
  ncthreeW/.pic={
    \draw (0,0.2)--(0,0)--(0.2,0)--(0.2,0.2);
	\draw (0.1,0)--(0.1,0.2);
  }
}

\tikzset{
  ncfour/.pic={
    \draw (0,0)--(0,0.2);
	\draw (0.1,0)--(0.1,0.2);
	\draw (0.2,0)--(0.2,0.2);
	\draw (0.3,0)--(0.3,0.2);
  }
}
\tikzset{
  ncfive/.pic={
    \draw (0,0)--(0,0.2);
	\draw (0.1,0)--(0.1,0.2);
	\draw (0.2,0)--(0.2,0.2);
	\draw (0.3,0)--(0.3,0.2);
	\draw (0.4,0)--(0.4,0.2);
  }
}

\tikzset{
  ncfourW/.pic={
    \draw (0,0.2)--(0,0)--(0.3,0)--(0.3,0.2);
	\draw (0.1,0)--(0.1,0.2);
	\draw (0.2,0)--(0.2,0.2);
  }
}
\tikzset{
  ncfiveW/.pic={
    \draw (0,0.2)--(0,0)--(0.4,0)--(0.4,0.2);
	\draw (0.1,0)--(0.1,0.2);
	\draw (0.2,0)--(0.2,0.2);
	\draw (0.3,0)--(0.3,0.2);
  }
}

\tikzset{
  nconeinsidetwoWW/.pic={
    \path (0,0) pic {nctwoWW}; \path (0.1,0.1) pic {ncone};
  }
}

\tikzset{
  nconeinsidethreeWW/.pic={
    \path (0,0) pic {ncthreeWW}; \path (0.1,0.1) pic {ncone};
  }
}
\tikzset{
  nconeinsidethreerightWW/.pic={
    \draw (0,0.2)--(0,0)--(0.3,0)--(0.3,0.2);
    \draw (0.1,0)--(0.1,0.2); \draw (0.2,0.1)--(0.2,0.3);
  }
}
\tikzset{
  nconeinsidethreeleftWW/.pic={
    \draw (0,0.2)--(0,0)--(0.3,0)--(0.3,0.2);
    \draw (0.2,0)--(0.2,0.2); \draw (0.1,0.1)--(0.1,0.3);
  }
}
\tikzset{
  nctwoWWW/.pic={
    \draw (0,0.2)--(0,0)--(0.3,0)--(0.3,0.2);
  }
}

\tikzset{
  nconeoneinsidetwoWW/.pic={
	\path (0,0) pic {ncone};
    \path (0.1,0) pic {nctwoWW}; 
	\path (0.2,0.1) pic {ncone};
  }
}

\usepackage{forest}
\forestset{
	decor/.style = {
		label/.expanded = {[inner sep = 0.2ex, font=\unexpanded{\tiny}]right:{$#1$}}
	},
	root/.style = {minimum size = 0.1ex},
	decorated/.style = {
		for tree = {
			circle, fill, inner sep = 0.3ex, minimum size = 1.ex,
			grow' = south, l = 0, l sep = 1.2ex, s sep = 0.7em,
			fit = tight, parent anchor = center, child anchor = center,
			delay = {decor/.option = content, content =}
		}
	},
	default preamble = {decorated, root},
	begin draw/.code={\begin{tikzpicture}[baseline={([yshift=-0.5ex]current bounding box.center)}]},
}

\newtheorem{theorem}{Theorem}[section]

\newtheorem{lemma}[theorem]{Lemma}
\newtheorem{proposition}[theorem]{Proposition}
\newtheorem{definition}{Definition}
\newtheorem{example}[theorem]{Example}
\newtheorem{remark}{Remark}

\DeclareMathOperator{\graft}{\triangleright}

\title{Translations of rough paths in combinatorial Hopf algebras}
\author{Ludwig Rahm\footnote{Department of Mathematical Sciences, Norwegian University of Science and Technology (NTNU), 7491 Trondheim, Norway. \texttt{ludwig.rahm@ntnu.no}.}}

\begin{document}

\maketitle

\begin{abstract}
We generalize Bruned et.~al.'s notion of translation in geometric and branched rough paths to a notion of translation in rough paths over any combinatorial Hopf algebra. We show that this notion of translation is equivalent to two bialgebras being in cointeraction, subject to certain additional conditions. We argue that reformulating translations in terms of substitutions, provides simpler conditions for the cointeraction formulation. For the special case where the translation can be obtained from a product, we show how to obtain a description of the dual coaction. As a concrete example, we describe translations in planarly branched rough paths.
\end{abstract}

\section{Introduction}
\label{sec:intro}

The notion of rough path grew out of Lyons' 1998 work on differential equations \cite{Lyons1998} 
\allowdisplaybreaks
\begin{align} 
\label{eq::Diffeq}
	dY_{st}=\sum_{i=1}^df_i(Y_{st})dX_t^i,
\end{align}
where $Y_{0s}: \mathbb{R} \to \mathbb{R}^d$ is an unknown path, the $f_i:\mathbb{R}^d \to \mathbb{R}^d$ are vector fields and $X_s: \mathbb{R} \to \mathbb{R}^d$ is some driving path. If the latter is sufficiently smooth, then it makes sense to integrate against it iteratively. One obtains a solution in terms of a sum of iterated integrals:
\allowdisplaybreaks
\begin{align*}
	Y_{st}=\sum_{n=1}^{\infty} \sum_{1 \leq i_1 \leq \cdots \leq i_n\leq n} \,
	\big(\idotsint\limits_{s\leq t_1\leq \dots\leq t_n\leq t} 
	f_{i_n}\cdots f_{i_1}dX_{t_1}^{i_1}\cdots dX_{t_n}^{i_n}\big).
\end{align*}
The theory of rough paths is based on the idea that if the path $X_t$ is not sufficiently smooth, then Equation \eqref{eq::Diffeq} requires some additional information to understand it in terms of its iterated integrals \cite{Gubinelli2003}. See, e.g., \cite{HairerKelly2012} for details. A rough path $\mathbb{X}_t$ over the path $X_t$ is the original path $X_t$ together with its --abstractly defined-- iterated integrals. This then defines the notion of a differential equation controlled by the rough path $\mathbb{X}_t$:
\begin{align}
\label{eq::ControlledDiffeq}
	dY_{st}=\sum_{i=1}^df_i(Y_{st})d\mathbb{X}_t^i.
\end{align}
In Lyon's work, a rough path $\mathbb{X}_t$ over $X_t$ is encoded as a path in the character group $(G,\otimes)$ of the shuffle Hopf algebra over $\mathbb{R}^d$ such that $ \mathbb{X}_{st}:= \mathbb{X}_{t} \otimes  \mathbb{X}^{-1}_{s}$ yielding 
\begin{align*}
	\langle \mathbb{X}_{st}, e_i \rangle = X_t^i-X_s^i.
\end{align*}
Here $e_i$ is the basis vector for the $i:th$ coordinate and the inverse $\mathbb{X}^{-1}_{s}=\mathbb{X}_{s} \circ S$, where $S$ is the antipode in the shuffle Hopf algebra. The evaluation $\langle \mathbb{X}_{st},e_{i_1}\cdots e_{i_n}\rangle$ plays the role of the --formally defined-- iterated integral
\begin{align*}
	\langle \mathbb{X}_{st},e_{i_1}\cdots e_{i_n}\rangle=
	\idotsint\limits_{s\leq t_1\leq \dots\leq t_n\leq t} dX_{t_1}^{i_1}\cdots dX_{t_n}^{i_n}.
\end{align*}
The righthand side has to be understood formally. Rough paths over the shuffle Hopf algebra are nowadays denoted (weak) \textit{geometric} rough paths. The concept of rough paths was extended by Gubinelli in \cite{Gubinelli2010} to encode more general integrals. He defined the notion of \textit{branched} rough path as paths in the character group over the Butcher--Connes--Kreimer Hopf algebra of non-planar rooted trees. This idea was then developed further by Curry et.~al.~in \cite{CurryEbrahimiFardManchonMuntheKaas2018}, where the concept of a rough path over any combinatorial Hopf algebra was defined. Inspired by Lie group integration theory \cite{IserlesMunthe-KaasNorsettZanna2000} and the notion of Lie--Butcher series \cite{LundervoldMunthe-Kaas2011} they furthermore showed how rough paths over the Munthe-Kaas--Wright Hopf algebra of planar rooted trees can be use to define solutions to rough differential equations on homogeneous spaces.\\

Rough path theory was generalized to the theory of regularity structures by Hairer in \cite{Hairer2013}. From this generalization follows a precise correspondence between notions from rough paths theory and from regularity structures. This correspondence was examined by Bruned et.~al.~in \cite{BrunedChevyrevFrizPreiss2017}, where they asked the question, whether renormalization of so-called models in regularity structures, studied comprehensively in \cite{BHZ2019}, has a corresponding analogue for rough paths. They showed that the answer is affirmative and established a correspondence between renormalization and translations of geometric as well as branched rough path.\\

The paper at hand aims at extending the notion of translation of geometric and branched rough paths, to translations of rough paths over any combinatorial Hopf algebra. Based on the properties of both geometric and branched rough paths translations, we propose a definition for translations of rough paths with respect to any combinatorial Hopf algebra. We then show that this definition can be understood via the dual map,  as two Hopf algebras in cointeraction that satisfy certain extra conditions. We furthermore show that translations can equivalently be understood as substitutions, better known in the context of Butcher's $B$-series, where the dual formulation in terms of cointeracting Hopf algebras is subject to simpler conditions. Translations in planarly branched rough paths are constructed and we describe how these translations affect the solution to rough differential equations on homogeneous spaces driven by a planarly branched rough path.

\medskip

The structure of the paper is as follows: In section \ref{section::Preliminaries}, we summarize the definitions and results that the present paper builds upon. In section \ref{section::Translations}, we define translations for rough paths over a combinatorial Hopf algebra and characterize the dual map as a cointeraction between two Hopf algebras. In section \ref{section::Substitution}, we show that translations can equivalently be thought of as substitutions. The substitution formulation provides a simpler way to describe the dual map. We give a formula to convert between the dual map for substitution and the dual map for translation. In section \ref{section::Algebraic}, we show how extra algebraic structure on the Hopf algebra that we take rough paths over can give us an explicit translation map, whose dual map can be described with coloured operads. In section \ref{section::Examples}, we apply the construction from the previous section to describe translations in planarly branched rough paths.

\section{Preliminaries} 
\label{section::Preliminaries}

We recall some definitions and results thereby fixing notations \cite{CartierPatras2021,ManchonHopf,Radford2012}. All algebraic structures are assumed to be defined over some fixed field $\mathbb{K}$ of characteristic zero.

\subsection{Combinatorial Hopf algebras}
\label{ssec:CHA}

A bialgebra $(V,\odot,\Delta,\eta,\epsilon)$ over the field $\mathbb{K}$ is a $\mathbb{K}$ vector space $V$ together with an associative multiplication, $\odot : V \otimes V \to V$, i.e., $x \odot (y \odot z)= (x \odot y) \odot z $, a coassociative coproduct, $\Delta: V \to V \otimes V$, i.e., $(Id \otimes \Delta)\Delta=(\Delta \otimes Id)\Delta$, a unit map, $\eta: \mathbb{K} \to V$, i.e., $\eta(1)\odot x=x$, and the counit, $\epsilon: V \to \mathbb{K}$, characterized by $(Id \otimes \epsilon)\Delta=Id=(\epsilon \otimes Id)\Delta$, satisfying the bialgebra relations:
\allowdisplaybreaks
\begin{align*}
	\Delta(x \odot y)=&\Delta(x) \odot \Delta(y),\\
	\epsilon(x)\epsilon(y)=&\epsilon(x \odot y),\\
	\Delta(\eta(x))=&(\eta \otimes \eta)\Delta_\mathbb{K} (x),\\
	Id_{\mathbb{K}} =& \epsilon \circ \eta.
\end{align*}
A graded bialgebra is called connected if the unit map, $\eta$, is an isomorphism between $\mathbb{K}$ and the set of degree zero elements. A Hopf algebra is defined as a bialgebra equipped with an anti-homomorphism $S: V \to V$ called the antipode satisfying
\allowdisplaybreaks
\begin{align*}
	\odot \circ (S \otimes Id)\Delta=\eta \circ \epsilon = \odot \circ ( Id \otimes S)\Delta.
\end{align*}
It is well-known that a connected and graded bialgebra is a Hopf algebra \cite{ManchonHopf}. 

\begin{definition}
Let $(\mathcal{H},\odot,\Delta,\eta,\epsilon)$ be a Hopf algebra. We say that the element $x \in \mathcal{H}$ is primitive if $\Delta(x)=1 \otimes x + x \otimes 1$, or grouplike if $\Delta(x)=x \otimes x$.
\end{definition}

\begin{definition} 
\label{def::cointeraction}
We say that two bialgebras $(A,\odot_{\scriptscriptstyle{A}},\Delta_{\scriptscriptstyle{A}},\epsilon_{\scriptscriptstyle{A}},\eta_{\scriptscriptstyle{A}})$, $(B,\odot_{\scriptscriptstyle{B}},\Delta_{\scriptscriptstyle{B}},\epsilon_{\scriptscriptstyle{B}},\eta_{\scriptscriptstyle{B}})$ are in cointeraction if $B$ is coacting on $A$ via a map $\rho: A \to B \otimes A$ that satisfies:
\allowdisplaybreaks
\begin{align*}
	\rho(1_A)=&1_B \otimes 1_A, \\
	\rho(x \odot_{\scriptscriptstyle{A}} y)=
	&\rho(x) (\odot_{\scriptscriptstyle{B}} \otimes \odot_{\scriptscriptstyle{A}}) \rho(y), \\
	(Id \otimes \epsilon_{\scriptscriptstyle{A}})\rho=&1_B \epsilon_{\scriptscriptstyle{A}}, \\
	(Id \otimes \Delta_{\scriptscriptstyle{A}})\rho=&m_{\scriptscriptstyle{B}}^{1,3}(\rho \otimes \rho)\Delta_{\scriptscriptstyle{A}},
\end{align*}
where 
\allowdisplaybreaks
\begin{align*}
	m_{\scriptscriptstyle{B}}^{1,3}(a \otimes b \otimes c \otimes d)
	=a \odot_{\scriptscriptstyle{B}}c \otimes b \otimes d.
\end{align*}
\end{definition}

For our purposes in this paper, we need the definition of a combinatorial Hopf algebra given by Curry et.~al.~in \cite{CurryEbrahimiFardManchonMuntheKaas2018}.

\begin{definition}
\label{def:combHopfAlg}
A combinatorial Hopf algebra $(V,\odot,\Delta,\eta,\epsilon)$ is a graded connected Hopf algebra $V = \oplus_{n=0}^{\infty}V_n$ over a field $\mathbb{K}$ of characteristic zero, together with a basis $\mathcal{B}=\cup_{n \geq 0} \mathcal{B}_n$ of homogeneous elements, such that:
\begin{enumerate}
	\item There exists two positive constants $B$ and $C$ such that the dimension of $V_n$ is bounded by $BC^n$.
	\item The structure constants $c_{x y}^{z}$ and $c_{z}^{x y}$ of the product respectively the coproduct, defined for all elements $x,y,z \in \mathcal{B}$ by
\allowdisplaybreaks
\begin{align*}
	x \odot y =& \sum_{z \in \mathcal{B}} c_{x y}^{z} z, \\
	\Delta(z) =& \sum_{x,y \in \mathcal{B}} c_{z}^{x y} x \otimes y,
\end{align*}
are non-negative integers.
\end{enumerate}
We furthermore say that $V$ is non-degenerate if $\mathcal{B} \cap \text{Prim}(V)=\mathcal{B}_1$.
\end{definition}

\begin{definition} \label{def::RoughPath}
Let $\mathcal{H}=\oplus_{n\geq 0} \mathcal{H}_n$ be a commutative graded Hopf algebra with unit $1$, and let $\gamma \in (0,1]$. Suppose that $\mathcal{H}$ is endowed with a basis $\mathcal{B}$ making it combinatorial and non-degenerate in the sense of Definition \ref{def:combHopfAlg}. A $\gamma$-regular $\mathcal{H}$-rough path is a two-parameter family $\mathbb{X}=(\mathbb{X}_{st})_{s,t \in \mathbb{R}}$ of linear forms on $\mathcal{H}$ such that $\langle \mathbb{X}_{st},1\rangle =1$ and: 
\begin{enumerate}
	\item For any $s,t \in \mathbb{R}$ and any $x,y \in \mathcal{H}$, the following identity holds
\allowdisplaybreaks
\begin{align*}
	\langle \mathbb{X}_{st},x \odot y \rangle 
	= \langle \mathbb{X}_{st},x \rangle\langle \mathbb{X}_{st}, y \rangle.
\end{align*}
	\item For any $s,t,u \in \mathbb{R}$, Chen's lemma holds
\allowdisplaybreaks
\begin{align*}
	\mathbb{X}_{su} \ast \mathbb{X}_{ut}=\mathbb{X}_{st},
\end{align*}
where $\ast$ is the convolution product for linear forms on $\mathcal{H}$, defined in terms of the coproduct on $\mathcal{H}$.
	\item For any $n \geq 0$ and any $x \in \mathcal{B}_n$, we have estimates
\allowdisplaybreaks
\begin{align*}
	\sup_{s \neq t} \frac{|\langle \mathbb{X}_{st},x \rangle |}{|t-s|^{\gamma |x|}} < \infty,
\end{align*}
where $|x|=n$ denotes the degree of the element $x \in \mathcal{B}_n$.
\end{enumerate}
\end{definition}

In any combinatorial Hopf algebra $\mathcal{H}$, we define the inverse-factorial character $q: \mathcal{H} \to \mathbb{R}$ by:
\allowdisplaybreaks
\begin{align*}
	q(x)=& 1, \quad x \in \mathcal{B}_1, \\
	q(y)=&\frac{1}{2^{|y|}-2}q(y_{(1)'})q(y_{(2)'}),
\end{align*}
where we use the Sweedler notation $\Delta'(y)= y_{(1)'} \otimes y_{(2)'}$ for the reduced coproduct $\Delta'(y)=\Delta(y)-1 \otimes y - y \otimes 1$ of $\mathcal{H}$. For $y \in \mathcal{B}$ we shall also use the notation
\allowdisplaybreaks
\begin{align*}
	y ! = \frac{1}{q(y)}.
\end{align*}

\begin{theorem} [\cite{CurryEbrahimiFardManchonMuntheKaas2018}] \label{Thm::estimate}
Let $\gamma \in (0,1]$, and let $N=\frac{1}{\gamma}$. Let $\mathbb{X}$ be a $\gamma$-regular $\mathcal{H}$-rough path. Then there exists a positive constant $c$ such that:
\allowdisplaybreaks\begin{align*}
|\langle \mathbb{X}_{st}, x \rangle| \leq c^{|x|}q_{\gamma}(x)|t-s|^{\gamma |x|},
\end{align*}
for any $x \in \mathcal{B}$, where
\allowdisplaybreaks\begin{align*}
q_{\gamma}(x)=& \begin{cases*}
q(x), \quad |x| \leq N \\
\frac{1}{2^{\gamma |x|}-2}q_{\gamma}(x')q_{\gamma}(x''), \quad |x| > N
\end{cases*}.
\end{align*}
\end{theorem}

\subsection{Trees and forests}
\label{ssec:trees}

A rooted tree is a connected graph without cycles, together with a distinguished vertex called the root. We say that the rooted tree is non-planar if it is not endowed with a preferred embedding into the plane. It is called planar if it is endowed with such an embedding into the plane. We will draw rooted trees with the root at the top. The two trees
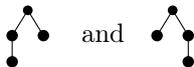
\allowdisplaybreaks
\begin{align*}
	\Forest{[[][[]]]} \quad \text{and} \quad  \Forest{[[[]][]]}
\end{align*}
are isomorphic as graphs via an isomorphism that sends the root to the root, hence they are equal as non-planar trees. However, considering embeddings into the plane makes them different, hence they are not equal as planar rooted trees. An unordered sequence of non-planar rooted trees is called a non-planar forest. An ordered sequence of planar rooted trees is called an ordered forest. We say that a rooted tree/forest is decorated by the set $\mathcal{C}$ if there is a map from the vertices of the rooted tree/forest to the set $\mathcal{C}$. A decoration of a vertex will be drawn by writing the decoration next to the vertex. We will denote by $\mathcal{T}_{\mathcal{C}}$ the vector space of non-planar rooted trees decorated by $\mathcal{C}$. The vector space of decorated planar rooted trees is denoted by $\mathcal{PT}_{\mathcal{C}}$. The vector space of decorated non-planar and ordered forests is denoted by $\mathcal{F}_{\mathcal{C}}$ respectively $\mathcal{OF}_{\mathcal{C}}$. \\

Non-planar rooted trees can be endowed with the grafting product $\curvearrowright: \mathcal{T}_{\mathcal{C}} \otimes \mathcal{T}_{\mathcal{C}} \to \mathcal{T}_{\mathcal{C}}$ given by defining $\tau_1 \curvearrowright \tau_2$ to be the sum of all rooted trees obtained by adding one edge from some vertex of $\tau_2$ to the root of $\tau_1$. The root of each of the rooted trees in the sum $\tau_1 \curvearrowright \tau_2$ is the root of $\tau_2$. The algebra $(\mathcal{T}_{\mathcal{C}},\curvearrowright)$ is the free pre-Lie algebra \cite{Burde2006,Cartier11,ChapotonLivernet2001,Livernet2006,Manchon2009,OudomGuin2008}. Pre-Lie algebras are defined by the (left) pre-Lie relation
\allowdisplaybreaks
\begin{align*}
	x \curvearrowright (y \curvearrowright z) - (x \curvearrowright y) \curvearrowright z 
	- y \curvearrowright (x \curvearrowright z) + (y \curvearrowright x) \curvearrowright z =0.
\end{align*}
The property of being a free pre-Lie algebra means that for any other pre-Lie algebra $(A,\diamond)$ and a map $\phi: \mathcal{C} \to A$, there exists a unique pre-Lie algebra morphism $\xi: \mathcal{T}_{\mathcal{C}} \to A$ such that $\xi(\bullet_{c})=\phi(c)$ for all $c \in \mathcal{C}$. \\

Non-planar forests can be endowed with a combinatorial Hopf algebra structure $\mathcal{H}_{BCK}^{\mathcal{C}}=(\mathcal{F}_{\mathcal{C}},\odot,\Delta_{BCK})$ called the Butcher--Connes--Kreimer Hopf algebra \cite{ConnesKreimer1998}. The commutative product $\odot$ is given by the disjoint union of two unordered sequences of rooted trees. The coproduct $\Delta_{BCK}$ is defined by so-called admissible edge cuts. Let $\tau \in \mathcal{T}_{\mathcal{C}} $ be a non-planar tree and let $c$ be a (possibly empty) subset of edges in $\tau$. We say that $c$ is an admissible edge cut if it contains at most one edge from each path in $\tau$ that starts in the root and ends in a leaf. Removing the edges in $c$ from $\tau$ produces several connected components, the connected component containing the root of $\tau$ will be denoted by $\mathrm{R}^c(\tau)$. The product of the remaining connected components will be denoted by $\mathrm{P}^c(\tau)$. The coproduct is then given by
\allowdisplaybreaks
\allowdisplaybreaks
\begin{align*}
	\Delta_{BCK}(\tau)=\sum_{c \text{ admissible cut}} \mathrm{P}^c(\tau) \otimes \mathrm{R}^c(\tau)
	+\tau \otimes 1
\end{align*}
on non-planar rooted trees, and extended to forests by
\allowdisplaybreaks
\allowdisplaybreaks
\begin{align*}
	\Delta_{BCK}(\tau_1\odot \cdots\odot \tau_n)=\Delta_{BCK}(\tau_1)\odot \dots \odot \Delta_{BCK}(\tau_n).
\end{align*}
A rough path in $\mathcal{H}_{BCK}$ is called a branched rough path \cite{Gubinelli2010}.\\

The undecorated non-planar forests can be endowed with another structure of bialgebra denoted $\mathcal{H}_{CEFM}=(\mathcal{F},\odot,\Delta_{CEFM})$ \cite{CalaqueEbrahimi-FardManchon2011}. The product $\odot$ is the same as for $\mathcal{H}_{BCK}$. The coproduct on the other hand $\Delta_{CEFM}$ is defined by contractions of rooted subtrees. Let $\tau \in \mathcal{T}$ be a non-planar rooted tree and let $(\tau_1,\ldots,\tau_n)$ be a spanning subforest of $\tau$, i.e., each $\tau_i$ is a rooted subtree of $\tau$ and each vertex of $\tau$ is contained in exactly one $\tau_i$. We denote by $\tau/(\tau_1,\ldots, \tau_n)$ the tree obtained by contracting each subtree to a single vertex. The coproduct, $\Delta_{CEFM}$, is then given by
\allowdisplaybreaks
\begin{align*}
	\Delta_{CEFM}(\tau)
	=\sum_{(\tau_1,\ldots,\tau_n) \atop \text{ spanning}\ \text{subforest}} 
	\tau_1 \odot \cdots \odot \tau_n \otimes \tau/(\tau_1, \ldots, \tau_n)
\end{align*}
and extended to forests multiplicatively
\allowdisplaybreaks
\begin{align*}
	\Delta_{CEFM}(\tau_1\odot \cdots \odot \tau_n)
	=\Delta_{CEFM}(\tau_1)\odot \cdots\odot \Delta_{CEFM}(\tau_n).
\end{align*}
Planar rooted trees can be endowed with the grafting product $\graft: \mathcal{PT}_{\mathcal{C}} \otimes \mathcal{PT}_{\mathcal{C}} \to \mathcal{PT}_{\mathcal{C}}$ given by defining $\tau_1 \graft \tau_2$ to be the sum of all rooted trees obtained by adding one edge from some vertex of $\tau_2$ to the root of $\tau_1$, such that the added edge is leftmost on the vertex in $\tau_2$ relative to the planar embedding. The root of each of the rooted trees in the sum $\tau_1 \graft \tau_2$ is the root of $\tau_2$. Let $(Lie(\mathcal{PT}_{\mathcal{C}}),[\cdot,\cdot])$ denote the free Lie algebra generated by $\mathcal{PT}_{\mathcal{C}}$ and extend $\graft: Lie(\mathcal{PT}_{\mathcal{C}}) \otimes Lie(\mathcal{PT}_{\mathcal{C}}) \to Lie(\mathcal{PT}_{\mathcal{C}})$ by the relations
\allowdisplaybreaks
\begin{align}
\begin{aligned} 
\label{postLierel}
	\omega_1 \graft [\omega_2,\omega_3]=
	&[\omega_1 \graft \omega_2,\omega_3] + [\omega_2,\omega_1 \graft \omega_3], \\
	[\omega_1,\omega_2] \graft \omega_3 =
	& \omega_1 \graft (\omega_2 \graft \omega_3)
	-(\omega_1 \graft \omega_2)\graft \omega_3 
	- \omega_2 \graft (\omega_1 \graft \omega_3 ) 
	+ (\omega_2 \graft \omega_1) \graft \omega_3.
\end{aligned}
\end{align}
Then $(Lie(\mathcal{PT}_{\mathcal{C}}),\graft,[\cdot,\cdot])$ is the free post-Lie algebra \cite{CurryEbrahimi-FardMunthe-Kaas2017,Ebrahimi-FardLundervoldMunthe-Kaas2014,LundervoldMunthe-Kaas2011,Munthe-KaasLundervold2012,Silva2018}. Post-Lie algebras are defined by the two relations \eqref{postLierel} above.\\

Ordered forests can be endowed with a combinatorial Hopf algebra structure $\mathcal{H}_{MKW}^{\mathcal{C}}=(\mathcal{OF}_{\mathcal{C}},\shuffle,\Delta_{MKW})$ known as the Munthe-Kaas--Wright Hopf algebra \cite{Munthe-KaasWright2008}. The commutative product $\shuffle$ is given by the sum of all ways to merge two ordered sequences of rooted trees into one sequence, so that the order from the two original sequences is preserved. The coproduct, $\Delta_{MKW}$, is defined by admissible left edge cuts. Let $\tau \in \mathcal{PT}_{\mathcal{C}}$ be a planar rooted tree and let $c$ be a (possibly empty) subset of edges in $\tau$. We say that $c$ is an admissible left edge cut if it contains at most one edge from each path in $\tau$ from the root to a leaf. Furthermore if $e$ is an edge in $c$, then every edge outgoing from the same vertex as $e$ and that is to the left of $e$ in the planar embedding, is also in $c$. Removing the edges in $c$ from $\tau$ produces several connected components, the one containing the root of $\tau$ will be denoted by $\mathrm{R}^c(\tau)$. Connected components that are cut off from the same vertex will be concatenated to an ordered forest respecting the order, and then the resulting ordered forests will be shuffled together, which is denoted by $\mathrm{P}^c(\tau)$. The coproduct $\Delta_{MKW}$ is defined by
\allowdisplaybreaks
\begin{align*}
	\Delta_{MKW}(\tau)=\sum_{c \text{ left admissible cut}} 
	\mathrm{P}^c(\tau) \otimes \mathrm{R}^c(\tau) + \tau \otimes 1
\end{align*}
on planar rooted trees. It is extended to ordered forests by
\allowdisplaybreaks\begin{align*}
	\Delta_{MKW}(\omega)=(Id \otimes B^-)\Delta_{MKW}(B^+(\omega)),
\end{align*}
where $B^+:\mathcal{OF}_{\mathcal{C}} \to \mathcal{PT}_{\mathcal{C}}$ is given by grafting all trees in the input sequence onto the same root in such a way that the planar embedding represents the order of the sequence and $B^-: \mathcal{PT}_{\mathcal{C}} \to \mathcal{OF}_{\mathcal{C}}$ is the inverse map. A rough path in $\mathcal{H}_{MKW}$ is called a planarly branched rough path  \cite{CurryEbrahimiFardManchonMuntheKaas2018}. The product dual to $\Delta_{MKW}$ is called the planar Grossman--Larson product, given by
\allowdisplaybreaks
\begin{align} \label{eq::PlanarGL}
	\omega_1 \ast \omega_2 = (\omega_1)_{(1)}((\omega_1)_{(2)}\graft \omega_2  ),
\end{align}
where $\Delta_{\shuffle}(\omega)=\omega_{(1)}\otimes \omega_{(2)}$ is the Sweedler notation for the unshuffle coproduct on words, and the planar grafting is extended to forests by
\allowdisplaybreaks
\begin{align*}
	\tau \graft \tau_1\cdots \tau_n =
	&(\tau \graft \tau_1)\tau_2 \cdots \tau_n 
		+ \tau_1(\tau \graft \tau_2)\tau_3 \cdots \tau_n 
		+ \dots + \tau_1 \cdots \tau_{n-1}(\tau \graft \tau_n),\\
	(\omega \tau) \graft \omega'=
	& \omega \graft (\tau \graft \omega')- (\omega \graft \tau) \graft \omega',
\end{align*}
for $\omega,\omega'$ ordered forests and $\tau,\tau_1,\dots,\tau_n$ planar rooted trees.\\

Ordered forests, together with planar grafting extended to forests and non-commutative associative concatenation, form the free $D$-algebra \cite{LundervoldMunthe-Kaas2011,Munthe-KaasLundervold2012,Munthe-KaasWright2008}. A unital associative algebra $(A,\cdot)$ with a non-associative product $\graft$, is a $D$-algebra if
\begin{align*}
1 \graft a =&\ a, \\
a \graft x \in&\ \mathcal{D}(A), \\
x \graft (a \graft b) =&\ (x \cdot a)\graft b + (x \graft a) \graft b,
\end{align*}
for $a,b \in A$ and $x \in \mathcal{D}(A)$, where 
$$
	\mathcal{D}(A)=\{x \in A: x \graft (a \cdot b)= (x \graft a)\cdot b + a \cdot (x \graft b), \; \forall a,b\in A \}
$$
denotes the set of derivations in $A$.

\subsection{Rough differential equations on a homogeneous space}
\label{ssec:RDEhspace}

We recall the notion of rough differential equations on homogeneous spaces together with the solutions, as described in \cite{CurryEbrahimiFardManchonMuntheKaas2018}. Let $X_t : \mathbb{R} \to \mathbb{R}^d$ be a $\gamma$-Hölder continuous path. We are interested in the equation
\begin{align} 
\label{eq::PlanarRoughDiffEq}
	dY_{st}=\sum_{i=1}^d \#f_i(Y_{st})dX_t^i,
\end{align}
with initial condition $Y_{ss}=y$. The unknown is a path $Y_s : \mathbb{R} \to \mathcal{M}$, that maps $t$ to $Y_{st}$, where the homogeneous space $\mathcal{M}$ is a manifold together with a transitive action by a Lie group $G$:
\begin{align*}
(G \times \mathcal{M}) \ni (g,e) \mapsto g.e \in \mathcal{M}.
\end{align*}
The elements $f_i : \mathcal{M} \mapsto Lie(G)$, $i=1,\ldots,d$, are smooth maps into the Lie algebra of $G$. The map $\# : C^{\infty}(\mathcal{M},Lie(G)) \mapsto C^{\infty}(\mathcal{M},T\mathcal{M})$ is given by
\begin{align*}
\#g(y)=\frac{d}{dt}_{|_{t=0}}\exp(tg(y)).y \in T_y\mathcal{M},
\end{align*}
and defines the vector fields $\#f_i$. \\

Let $\mathcal{U}(Lie(G))$ denote the universal enveloping algebra of the Lie algebra $Lie(G)$. Then $C^{\infty}(\mathcal{M},\mathcal{U}(Lie(G)))$ together with the pointwise associative product in $\mathcal{U}(Lie(G))$ and the product $\graft$ given by
\begin{align*}
	f \graft g = \frac{d}{dt}_{|_{t=0}}g(\exp(tf(x)).x)
\end{align*}
is a $D$-algebra. Let $f=(f_1,\dots,f_d)$ be a list of elements of $C^{\infty}(\mathcal{M},\mathcal{U}(Lie(G)))$, then the universality property of the free $D$-algebra implies that there exists a unique $D$-algebra morphism $\mathcal{F}_f: \mathcal{OF}_{ \{1,\dots,d \} } \to C^{\infty}(\mathcal{M},\mathcal{U}(Lie(G)))$ given by $\mathcal{F}_f(\bullet_i)=f_i$, for $i=1,\dots,d$.

\begin{definition} \label{def::PlanarRoughSolution}
A formal solution to Equation \eqref{eq::PlanarRoughDiffEq} is given by
\begin{align*}
	Y_{st}=\#\mathcal{F}_f(\mathbb{Y}_{st})(y),
\end{align*}
where
\begin{align*}
	\mathbb{Y}_{st}=\sum_{\omega \in \mathcal{OF}_{ \{1,\dots,d\} }}\langle \mathbb{X}_{st},\omega \rangle \omega,
\end{align*}
and where $\mathbb{X}_{st}$ is any planarly branched rough path such that $\langle \mathbb{X}_{st},\bullet_i\rangle = X_t^i-X_s^i,$ for $i=1,\dots,d$.
\end{definition}

\subsection{Translations in geometric- and branched rough paths} \label{ssection::GeoBranchedTranslations}

We recall the notion of translation of rough paths from Bruned et.~al.~\cite{BrunedChevyrevFrizPreiss2017}.

\smallskip

Let $(T(\mathcal{C}),\shuffle,\Delta_{\odot})$ denote the shuffle Hopf algebra of non-commutative words with letters from the finite alphabet $\mathcal{C}$, with deconcatenation as coproduct. Let $\mathcal{B}$ be a basis of Lie polynomials in $T(\mathcal{C})$ such that the Hopf algebra is combinatorial and non-degenerate. Then a geometric rough path is a $T(\mathcal{C})$-rough path. We shall denote the letters in $\mathcal{C}=\mathcal{B}_1$ by $e_i, \; i=1,\ldots,n$, for $|\mathcal{C}|=n$. Let $(T(\mathcal{C})^{\ast},\odot,\Delta_{\shuffle}  )$ denote the graded dual Hopf algebra to $(T(\mathcal{C}),\shuffle,\Delta_{\odot}  )$, it can be identified with $T(\mathcal{C})$ by using the canonical dual basis. We write $\overline{T(\mathcal{C})^{\ast}}$ for its completion. The completed dual $(\overline{T(\mathcal{C})^{\ast}},\odot,\Delta_{\shuffle}  )$ can be equipped with a Hopf-type algebra structure. Note that this is not exactly a Hopf algebra as $\Delta_{\shuffle}$ does not map $\overline{T(\mathcal{C})^{\ast}}$ into $(\overline{T(\mathcal{C})^{\ast}})^{\otimes 2}$, but rather into $(T(\mathcal{C})^{\ast})^{\overline{\otimes} 2} \simeq \prod_{m,n=0}^{\infty}T(\mathcal{C})_m \otimes T(\mathcal{C})_n$. Then the infinitesimal characters of $(T(\mathcal{C}),\shuffle,\Delta_{\odot})$ are primitive in $(\overline{T(\mathcal{C})^{\ast}},\odot,\Delta_{\shuffle}  )$, and the characters are grouplike.\\

A translation $T_v:\overline{T(\mathcal{C})^{\ast}} \to \overline{T(\mathcal{C})^{\ast}} $, defined for a collection $v=(v_1,\ldots,v_n)$ of elements that are primitive with respect to $\Delta_{\shuffle}$, is the unique map given by
\allowdisplaybreaks
\begin{align*}
	T_v(e_i)=e_i+v_i
\end{align*}
and extended to be a continuous algebra morphism with respect to the concatenation product. The following properties hold \cite{BrunedChevyrevFrizPreiss2017}:
\begin{itemize}
\item $T_v$ maps primitive elements (infinitesimal characters) to primitive elements, and grouplike elements (characters) to grouplike elements.
\item $T_v \circ T_u = T_{v+T_v(u)}$.
\item $T_v$ maps rough paths to rough paths.
\item $T_v$ can dually be described by a coaction $\rho: T(\mathcal{C})\to S(T(\mathcal{C})_{in}\times \mathcal{C})\otimes T(\mathcal{C})$, where $T(\mathcal{C})_{in}$ are the non-trivial indecomposable elements, as $\langle T_v(\chi),x\rangle = \langle v \otimes \chi,\rho(x)\rangle$.
\end{itemize}

Let $\mathcal{H}_{BCK}^\mathcal{C}=(\mathcal{F}_{\mathcal{C}},\odot,\Delta_{BCK})$ be the Butcher--Connes--Kreimer Hopf algebra of non-planar rooted trees, and let $(\mathcal{H}_{BCK}^{\mathcal{C}})^{\ast}=(\mathcal{F}_\mathcal{C},\ast,\Delta_{\odot} )$ be the dual graded Hopf algebra. The basis $\mathcal{B}$ is given by the forests. In particular, $\mathcal{B}_1=\mathcal{C}$ is given by single vertex trees decorated by $\mathcal{C}$. Denote the completed dual by $\overline{(\mathcal{H}_{BCK}^{\mathcal{C}})^{\ast}}=(\overline{\mathcal{F}_{\mathcal{C}}},\ast,\Delta_{\odot})$, where $\Delta_{\odot}:\overline{\mathcal{F}_{\mathcal{C}}} \to (\mathcal{F}_{\mathcal{C}})^{\overline{\otimes} 2}$. Note that, as in the geometric case, this is not exactly a Hopf algebra. \\

One may now attempt to define a translation $M_v: \overline{(\mathcal{H}_{BCK}^{\mathcal{C}})^{\ast}} \to \overline{(\mathcal{H}_{BCK}^{\mathcal{C}})^{\ast}}$ by
\begin{align*}
	M_v(e_i)=e_i+v_i,
\end{align*}
for $v_i$ primitive, and extend this as a continuous algebra morphism with respect to $\ast$. It turns out that this construction does not admit a unique extension. \\

One algebraic structure that does extend uniquely from single-vertex trees to trees is the pre-Lie algebraic structure. Extending the map $M_v$ from the previous paragraph to be a pre-Lie algebra morphism on rooted trees, and a $\ast$ morphism on forests, gives us the notion of translation in branched rough paths from \cite{BrunedChevyrevFrizPreiss2017}. Then $M_v$ has the properties:
\begin{itemize}
\item $M_v$ maps primitive elements to primitive elements, and grouplike elements to grouplike elements.
\item $M_v \circ M_u = M_{v+M_v(u)}$.
\item $M_v$ maps rough paths to rough paths.
\item $M_v$ can dually be described by a coaction $\rho: \mathcal{H}^{\mathcal{C}}_{BCK}\to S((\mathcal{H}^{\mathcal{C}}_{BCK})_{in}\times \mathcal{C})\otimes \mathcal{H}^{\mathcal{C}}_{BCK}$, where $(\mathcal{H}^{\mathcal{C}}_{BCK})_{in}$ are the non-trivial indecomposable elements, as $\langle M_v(\chi),x\rangle = \langle v \otimes \chi,\rho(x)\rangle$.
\end{itemize}

\section{Translations in rough paths} 
\label{section::Translations}

We propose to use the properties from Section \ref{ssection::GeoBranchedTranslations} as the definition for translations in rough paths over any combinatorial Hopf algebra. Before we write down the definition, we introduce some notation.\\

Let $(\mathcal{H},\odot,\Delta,\eta,\epsilon)$ be a non-degenerate combinatorial Hopf algebra with basis $\mathcal{B}$. Let $(\mathcal{H}^{\ast},\ast,\Delta_{\odot})$ denote the graded dual space, with convolution product $\ast$ dual to $\Delta$ and coproduct $\Delta_{\odot}$ dual to $\odot$. Identify $\mathcal{H}$ with $\mathcal{H}^{\ast}$ via the dual basis. Let $(\overline{\mathcal{H}^{\ast}},\ast,\Delta_{\odot})$ be the completed dual equipped with a Hopf-type algebra structure. Let $\mathcal{H}_{in}$ denote the indecomposable elements of $\ker(\epsilon)$. Let $\mathcal{H}_{in}\times \mathcal{B}_1$ denote vector space of pairs $(v_i,e_i), \; v_i \in \mathcal{H}_{in}, \; e_i \in \mathcal{B}_1$ where the vector space structure is given by linearity in the first component. Let $(S(\mathcal{H}_{in}\times \mathcal{B}_1),\centerdot,\Delta_{\centerdot})$ denote the free cofree unital commutative co-commutative Hopf algebra. Seeing $S(\mathcal{H}_{in}\times \mathcal{B}_1)$ as commutative polynomials in elements from $\mathcal{H}_{in}\times \mathcal{B}_1$, let it be graded by degree of the polynomials. Identify the graded dual space $S(\mathcal{H}_{in}\times \mathcal{B}_1)^{\ast}$ with $S(\mathcal{H}_{in} \times \mathcal{B}_1)$ by using the dual basis. Then $(S(\mathcal{H}_{in}\times \mathcal{B}_1)^{\ast},\centerdot,\Delta_{\centerdot})$ is a Hopf algebra, and the completion $(\overline{S(\mathcal{H}_{in}\times \mathcal{B}_1)^{\ast}},\centerdot,\Delta_{\centerdot})$ is a Hopf-type algebra such that
\begin{align*}
\langle x \centerdot y, z\rangle = \langle x \otimes y, \Delta_{\centerdot}(z)\rangle,
\end{align*}
for $x,y \in\overline{S(\mathcal{H}_{in}\times \mathcal{B}_1)^{\ast}}$ and $z \in S(\mathcal{H}_{in}\times \mathcal{B}_1)$. Define the map $\exp^{\centerdot}: \overline{S(\mathcal{H}_{in}\times \mathcal{B}_1)^{\ast}} \to \overline{S(\mathcal{H}_{in}\times \mathcal{B}_1)^{\ast}}$ by
\begin{align*}
\exp^{\centerdot}(x)=1+x+\frac{1}{2!}(x \centerdot x)+\frac{1}{3!}(x \centerdot x \centerdot x)+\cdots.
\end{align*}
Then $\exp^{\centerdot}$ maps primitive elements of $\overline{S(\mathcal{H}_{in}\times \mathcal{B}_1)^{\ast}}$ into characters over $S(\mathcal{H}_{in}\times \mathcal{B}_1)$ and satisfies the identity
\begin{align*}
\exp^{\centerdot}(x+y)=\exp^{\centerdot}(x)\centerdot \exp^{\centerdot}(y).
\end{align*}
For $v=\{v_1,\dots,v_n\}$ a set of primitive elements in $\overline{\mathcal{H}^{\ast}}$, with $n=|\mathcal{B}_1|$, denote
\begin{align*}
e^v=\exp^{\centerdot}\Big(\sum_{i=1}^n (v_i,e_i) \Big) \in \overline{S(\mathcal{H}_{in}\times \mathcal{B}_1)^{\ast}},
\end{align*}
this defines a bijection between the set of characters over $S(\mathcal{H}_{in}\times \mathcal{B}_1)$ and the set of possible parameters $v=\{v_1,\dots,v_n\}$. We are now ready to define a translation by $v$.

\begin{definition} \label{def::Translation}
A family of algebra morphisms $T_v: \overline{\mathcal{H}^{\ast}} \to \overline{\mathcal{H}^{\ast}}$ is a translation if
\begin{enumerate}
\item $T_v(e_i)=e_i+v_i$ for every $e_i \in \mathcal{B}_1$ and some $v=\{v_1,\dots,v_n\}$, $v_i\in \overline{\mathcal{H}^{\ast}}$ primitive.

\item $T_v \circ T_u = T_{v+T_v(u)}$, where $T_v(u)=\{T_v(u_1),\dots,T_v(u_n)\}$.

\item For each $\mathcal{H}$-rough path $\mathbb{X}_{st}$, the pointwise translation $T_v(\mathbb{X}_{st})=T_v(\mathbb{X})_{st}$ is a $\mathcal{H}$-rough path:
\begin{enumerate}

\item $T_v$ maps characters to characters.

\item $T_v$ is a morphism with respect to the convolution product of $\overline{\mathcal{H}^*}$.

\item The bound \allowdisplaybreaks\begin{align*}
\sup_{s \neq t} \frac{|\langle T_v(\mathbb{X}_{st}),x \rangle |}{|t-s|^{\gamma |x|}} < \infty
\end{align*}
holds.
\end{enumerate}
\item \label{item:coaction} There exists a coaction $\rho_T : \mathcal{H} \to S(\mathcal{H}_{in} \times \mathcal{B}_1)\otimes \mathcal{H}$ such that $\langle T_v(\chi),x\rangle = \langle e^v \otimes \chi,\rho_T(x)\rangle$.
\end{enumerate}
\end{definition}

\begin{remark}
Similar axioms were considered in \cite{BellingeriFrizPaychaPreiss2021}.
\end{remark}

By the property that $T_v$ maps rough paths to rough paths, we get that $\rho_T$ gives a cointeraction between $(S(\mathcal{H}_{in} \times \mathcal{B}_1),\centerdot,\Delta_{\centerdot})$ and $(\mathcal{H},\odot,\Delta)$.

\begin{lemma} \label{Lemma:characterevaluation}
If $\langle \chi,x \rangle = \langle \chi,y \rangle$ for every character $\chi$, then $x=y$.
\end{lemma}

\begin{proof}
Characters are determined by their value on indecomposable elements. Suppose that $x \neq y$, then there is an indecomposable element $z$ that appears a different number of times in a factorization of $x$ compared to a factorization of $y$. Generate a new character $\chi'$ that evaluates to the same value as $\chi$ on all indecomposable elements except $z$, and to a different value than $\chi$ on $z$. Then the character property implies that $\langle \chi',x \rangle \neq \langle \chi',y\rangle$.
\end{proof}

\begin{proposition} \label{Thm::Cointeraction}
The Hopf algebra $(S(\mathcal{H}_{in} \times \mathcal{B}_1),\centerdot,\Delta_{\centerdot})$ is in cointeraction with $(\mathcal{H},\odot,\Delta)$ by the coaction $\rho_T$ specified in Definition \ref{def::Translation}, property \ref{item:coaction}.
\end{proposition}

\begin{proof}
We have that $\rho_T(1) = 1 \otimes 1$ by the assumption that $T_v$ is an algebra morphism. \\

Let $x,y \in \mathcal{H}$ and let $\chi$ be a character, then since $T_v$ maps characters to characters:
\allowdisplaybreaks\begin{align*}
\langle e^v \otimes \chi,\rho_T(x \odot y) \rangle =& \langle T_v(\chi),x \odot y \rangle\\
=& \langle T_v(\chi),x\rangle \langle T_v(\chi),y \rangle \\
=& \langle e^v \otimes \chi, \rho_T(x)\rangle \langle e^v \otimes \chi ,\rho_T(y)\rangle,
\end{align*}
which implies
\allowdisplaybreaks\begin{align*}
\rho_T(x \odot y)=\rho_T(x)\odot \rho_T(y).
\end{align*}
The identitiy $(Id \otimes \epsilon)\rho_T = 1\epsilon$ follows from $\langle e^v \otimes 1, \rho_T(x) \rangle =0$ whenever $x \neq 1$. Where the unit $1 \in \overline{\mathcal{H}^{\ast}}$ is the same as the counit of $(\mathcal{H},\odot,\Delta)$, by the dual basis identification. \\

Lastly we need to prove the identity $(Id \otimes \Delta)\rho_T = m^{1,3}(\rho_T \otimes \rho_T)\Delta$. Recall that $T_v$ maps rough paths to rough paths, meaning that it satisfies $T_v(\psi)\ast T_v(\chi)=T_v(\psi \ast \chi)$ for $\psi,\chi$ characters. This means:
\allowdisplaybreaks\begin{align*}
\langle e^v \otimes \psi \otimes \chi , (Id \otimes \Delta)\rho_T(x) \rangle =&\langle e^v \otimes \psi \ast \chi, \rho_T(x) \rangle \\
=&\langle T_v(\psi \ast \chi),x \rangle \\
=&\langle T_v(\psi)\ast T_v(\chi),x \rangle\\ 
=& \langle T_v(\psi) \otimes T_v(\chi),\Delta(x) \rangle \\
=&\langle e^v \otimes \psi \otimes e^v \otimes \chi, (\rho_T \otimes \rho_T)\Delta(x) \rangle.
\end{align*}
The property
\begin{align*}
(Id \otimes \Delta)\rho_T=m^{1,3} (\rho_T \otimes \rho_T)\Delta
\end{align*}
now follows from Lemma \ref{Lemma:characterevaluation} and the fact that $e^v$ is a character, as $e^v$ evaluated on the first component of the tensor, multiplied by $e^v$ evaluated on the third component of the tensor, is the same as $e^v$ evaluated on the product of the components. Hence we have proved all the properties of cointeraction and the theorem follows.
\end{proof}

Note that the proof does not require the use of the two central translation properties $T_v \circ T_u=T_{v+T_v(u)}$ and $T_v(e_i)=e_i+v_i$. The properties $1.$ and $2.$ from Definition \ref{def::Translation} corresponds to properties of the coaction. We describe these properties in the two following propositions.\\

Extend the map to $\rho_T: \mathcal{H}_{in}\times \mathcal{B}_1 \to S(\mathcal{H}_{in}\times \mathcal{B}_1) \otimes S(\mathcal{H}_{in}\times \mathcal{B}_1)$ by letting it act on the first component, and then to $S(\mathcal{H}_{in}\times \mathcal{B}_1)$ as an algebra morphism, i.e. let
\begin{align*}
\rho_T( (w_1,e_{i_1})\centerdot \dots \centerdot (w_k,e_{i_k}) )=(\rho_T(w_1),e_{i_1})\centerdot \dots \centerdot (\rho_T(w_k),e_{i_k}).
\end{align*}

\begin{proposition}
The coaction $\rho_T$, when extended to $S(\mathcal{H}_{in} \times \mathcal{B}_1)$, satisfies the identity
\allowdisplaybreaks
\begin{align} \label{eq::CoTranslationAssociative}
(Id \otimes \rho_T)\rho_T=m^{1,2}((Id \otimes \rho_T \otimes Id)((\Delta_{\centerdot} \otimes Id)\rho_T)),
\end{align}
where
\allowdisplaybreaks
\begin{align*}
m^{1,2}(x_1 \otimes x_2 \otimes x_3 \otimes x_4)=x_1 \centerdot x_2 \otimes x_3 \otimes x_4.
\end{align*}
\end{proposition}

\begin{proof}
\allowdisplaybreaks\begin{align*}
\langle e^v \otimes e^u \otimes \chi,(Id \otimes \rho_T)\rho_T(x)\rangle=&\langle e^v \otimes T_u(\chi),\rho_T(x)\rangle\\
=&\langle T_v(T_u(\chi)),x\rangle\\
=&\langle T_{v+T_v(u)}(\chi),x\rangle \\
=& \langle e^{v+T_v(u)}\otimes \chi, \rho_T(x)\rangle \\
=&\langle e^v \centerdot e^{T_v(u)}\otimes \chi, \rho_T(x)\rangle \\
=&\langle e^v \otimes T_v(u) \otimes \chi, (\Delta_{\centerdot} \otimes Id)\rho_T(x)\rangle \\
=&\langle e^v \otimes e^v \otimes e^u \otimes \chi, (Id \otimes \rho_T \otimes Id)((\Delta_{\centerdot} \otimes Id)\rho_T(x)) \rangle \\
=&\langle e^v \otimes e^u \otimes \chi, m^{1,2}((Id \otimes \rho_T \otimes Id)((\Delta_{\centerdot} \otimes Id)\rho_T(x)))\rangle.
\end{align*}
\end{proof}

\begin{remark}
The map $\rho_T$ is not coassociative, meaning that the identity
\begin{align}
	(Id \otimes \rho_T)\rho_T=(\rho_T\otimes Id)\rho_T \label{eq::Coassociative}
\end{align}
does not hold. We instead have the relation \eqref{eq::CoTranslationAssociative}, which can be understood as a shifted coassociativity. The $\rho_T$ in $(\rho_T\otimes Id)$ from the coassociativity relation \eqref{eq::Coassociative} will as input take a monomial in the product $\centerdot$, and evaluate on each factor of this monomial by the property of being a $\centerdot$-morphism. The relation \eqref{eq::CoTranslationAssociative} says that instead of letting $\rho_T$ evaluate on each factor of the input, we have to sum over all possible ways of letting $\rho_T$ evaluate on a subset of factors.

One way to informally think of this property is to see a translation $x \mapsto x+v$ as being a sum of an identity map $x \mapsto x$ and a substitution map $x \mapsto v$. The relation \eqref{eq::CoTranslationAssociative} can then be seen as the dual way to encode this sum. Factors that $\rho_T$ evaluates on will dually correspond to substitution and factors that $\rho_T$ does not evaluate on will dually correspond to the identity map.

We will elaborate on substitution maps and the relation to coassociativity in section \ref{section::Substitution}.
\end{remark}

\begin{proposition}
The coaction $\rho_T$ satisfies the identity
\allowdisplaybreaks
\begin{align*}
\langle e^v \otimes e_i, \rho_T(x)\rangle =\langle v_i+e_i,x\rangle,
\end{align*}
for $e_i \in \mathcal{B}_1$.
\end{proposition}

\begin{proof}
\allowdisplaybreaks\begin{align*}
\langle e_i+v_i,x\rangle =& \langle T_v(e_i),x\rangle \\
=&\langle e^v \otimes e_i, \rho_T(x)\rangle.
\end{align*}
\end{proof}

We have now seen that every translation gives a cointeraction between $S(\mathcal{H}_{in} \times \mathcal{B}_1)$ and $\mathcal{H}$. One may then ask whether every such cointeraction will give a translation. Let $(\mathcal{H},\odot,\Delta)$ be a non-degenerate combinatorial Hopf algebra that is in cointeraction with $S(\mathcal{H}_{in}\times \mathcal{B}_1)$ by a coaction $\rho_T: \mathcal{H} \mapsto S(\mathcal{H}_{in}\times \mathcal{B}_1) \otimes \mathcal{H}$. Define a map $T_v: \overline{\mathcal{H}^{\ast}}\mapsto \overline{\mathcal{H}^{\ast}}$ by $\langle T_v(y),x\rangle = \langle e^v \otimes y, \rho_T(x)\rangle$. Then
\begin{enumerate}
\item $T_v$ is an algebra morphism:
\allowdisplaybreaks
\begin{align*}
\langle T_v(a \ast b),x\rangle =& \langle e^v \otimes a \otimes b,(Id \otimes \Delta)\rho_T(x) \rangle \\
=&\langle e^v \otimes a \otimes b,m^{1,3}(\rho_T \otimes \rho_T)\Delta(x)\rangle \\
=&\langle e^v \otimes a \otimes e^v \otimes b, (\rho_T \otimes \rho_T)\Delta(x)\rangle \\
=&\langle T_v(a)\otimes T_v(b),\Delta(x)\rangle \\
=&\langle T_v(a)\ast T_v(b),x\rangle,
\end{align*}
for all $a,b \in \mathcal{H}^{\ast}$.
\item $T_v$ maps characters to characters:
\allowdisplaybreaks
\begin{align*}
\langle T_v(\chi),x\odot y\rangle =& \langle e^v \otimes \chi, \rho_T(x \odot y)\rangle \\
=&\langle e^v \otimes \chi,\rho_T(x)\odot \rho_T(y)\rangle \\
=&\langle e^v \otimes \chi,\rho_T(x)\rangle \langle e^v \otimes \chi,\rho_T(y)\rangle \\
=&\langle T_v(\chi),x\rangle\langle T_v(\chi),y\rangle. 
\end{align*}
\item The bound 
\allowdisplaybreaks\begin{align*}
\sup_{s \neq t} \frac{|\langle T_v(\mathbb{X}_{st}),x \rangle |}{|t-s|^{\gamma |x|}} < \infty
\end{align*}
follows from
\allowdisplaybreaks\begin{align*}
|\langle T_v(\mathbb{X}_{st}),x \rangle | =& | \langle e^v \otimes \mathbb{X}_{st},\rho_T(x)\rangle | \\
=&|\langle e^v,x_{(1)}\rangle ||\langle \mathbb{X}_{st},x_{(2)}\rangle| \\
\leq& |\langle e^v,x_{(1)}\rangle | c^{|x_{(2)}|}q_{\gamma}(x_{(2)})|t-s|^{\gamma |x_{(2)}|}
\end{align*}
and that $|\langle e^v,x_{(1)}\rangle |$ is finite and independent of $|t-s|$.
\item If $(Id \otimes \rho_T)\rho_T=m^{1,2}((Id \otimes \rho_T \otimes Id)((\Delta_{\centerdot} \otimes Id)\rho_T))$, then:
\allowdisplaybreaks\begin{align*}
\langle T_v(T_u(\chi)),x\rangle =& \langle e^v \otimes T_u(\chi),\rho_T(x)\rangle \\
=&\langle e^v \otimes e^u \otimes \chi, (Id \otimes \rho_T)\rho_T(x)\rangle \\
=&\langle e^v \otimes e^u \otimes \chi, m^{1,2}((Id \otimes \rho_T \otimes Id)((\Delta_{\centerdot} \otimes Id)\rho_T(x)))\rangle \\
=&\langle e^v \otimes e^v \otimes e^u \otimes \chi, (Id \otimes \rho_T \otimes Id)((\Delta_{\centerdot} \otimes Id)\rho_T(x)) \rangle \\
=&\langle e^v \otimes e^{T_v(u)} \otimes \chi, (\Delta_{\centerdot} \otimes Id)\rho_T(x)\rangle \\
=&\langle e^v \centerdot e^{T_v(u)}\otimes \chi, \rho_T(x)\rangle \\
=& \langle e^{v + T_v(u)} \otimes \chi, \rho_T(x)\rangle \\
=&\langle T_{v+T_v(u)}(\chi),x\rangle.
\end{align*}
\item If $\langle e^v \otimes e_i, \rho_T(x)\rangle =\langle v_i+e_i,x\rangle$, then:
\allowdisplaybreaks\begin{align*}
\langle T_v(e_i),x \rangle =& \langle e^v \otimes e_i, \rho_T(x)\rangle \\
=&\langle v_i+e_i,x\rangle,
\end{align*}
hence $T_v(e_i)=e_i+v_i$.
\end{enumerate}

In total, we get the following result.

\begin{theorem}
Let $(\mathcal{H},\odot,\Delta)$ be a non-degenerate combinatorial Hopf algebra that is in cointeraction with $(S(\mathcal{H}_{in}\times \mathcal{B}_1),\centerdot,\Delta_{\centerdot})$ by a coaction $\rho_T: \mathcal{H} \to S(\mathcal{H}_{in}\times \mathcal{B}_1) \otimes \mathcal{H}$, satisfying
\allowdisplaybreaks
\begin{align*}
(Id \otimes \rho_T)\rho_T=\ &m^{1,2}((Id \otimes \rho_T \otimes Id)((\Delta_{\centerdot} \otimes Id)\rho_T)), \\
\langle e^v \otimes e_i, \rho_T(x)\rangle =\ &\langle v_i+e_i,x\rangle.
\end{align*}
Then the dual map $T_v$ given by
\allowdisplaybreaks
\begin{align*}
\langle T_v(y),x\rangle = \langle e^v \otimes y,\rho_T(x)\rangle
\end{align*}
is a translation.
\end{theorem}

\section{Substitution in rough paths} 
\label{section::Substitution}

We find it useful to consider translations also as substitutions, as this ends up giving us simpler identities. By a substitution of rough paths, we mean the following.

\begin{definition}
Let $(\mathcal{H},\odot,\Delta)$ be a non-degenerate combinatorial Hopf algebra with basis $\mathcal{B}$. Let $(\mathcal{H}^{\ast},\ast,\Delta_{\centerdot})$ denote the graded dual space, with product $\ast$ dual to $\Delta$ and coproduct $\Delta_{\centerdot}$ dual to $\odot$. Identify $\mathcal{H}$ with $\mathcal{H}^{\ast}$ via the dual basis. Let $(\overline{\mathcal{H}^{\ast}},\ast,\Delta_{\odot})$ be the completed dual equipped with a Hopf-type algebra structure. A family of algebra morphisms $S_v: \overline{\mathcal{H}^{\ast}} \to \overline{\mathcal{H}^{\ast}}$ is a substitution if
\begin{enumerate}
\item for $v=\{v_1,\dots,v_n\}$, $v_i$ primitive, $S_v(e_i)=v_i$ for every $e_i \in \mathcal{B}_1$. 
\item $S_v \circ S_u = S_{S_v(u)}$, where $S_v(u)=\{S_v(u_1),\dots,S_v(u_n)\}$.
\item For each $\mathcal{H}$-rough path $\mathbb{X}_{st}$, the pointwise translation $S_v(\mathbb{X}_{st})=S_v(\mathbb{X})_{st}$ is a $\mathcal{H}$-rough path:
\begin{enumerate}
\item $S_v$ maps characters to characters.
\item $S_v$ is a morphism with respect to the convolution product of $\mathcal{H}$.
\item The bound \allowdisplaybreaks\begin{align*}
\sup_{s \neq t} \frac{|\langle S_v(\mathbb{X}_{st}),x \rangle |}{|t-s|^{\gamma |x|}} < \infty
\end{align*}
holds.
\end{enumerate}
\item There exists a coaction $\rho_S : \mathcal{H} \to S(\mathcal{H}_{in} \times \mathcal{B}_1)\otimes \mathcal{H}$ such that $\langle S_v(\chi),x\rangle = \langle e^v \otimes \chi,\rho_S(x)\rangle$.
\end{enumerate}
\end{definition}

Substitutions of rough paths are essentially the same as translations. Indeed, if $S_v$ is a substitution by $v=\{v_1,\ldots,v_n \}$ then $T_{v'}:=S_v$ is a translation by $v'=\{v_1-e_1,\dots,v_n-e_n \}$ and vice-versa. The condition $T_{v'}(e_i)=e_i+v_i'$ is clear. The condition $T_{v'} \circ T_{u'} = T_{v'+T_{v'}(u')}$ can be seen by the computation
\allowdisplaybreaks\begin{align*}
T_{v'} \circ T_{u'}=&S_v \circ S_u \\
=&S_{S_v(u)}\\
=&S_{T_{v'}(u)} \\
=&T_{(T_{v'}(u))'},
\end{align*}
where 
\allowdisplaybreaks\begin{align*}
(T_{v'}(u))'=&\{T_{v'}(u_1)-e_1,\dots,T_{v'}(u_n)-e_n  \} \\
=&\{T_{v'}(u'_1 +e_1)-e_1,\dots,T_{v'}(u'_n+e_n)-e_n  \} \\
=&\{T_{v'}(u'_1)+v'_1+e_1-e_1,\dots,T_{v'}(u'_n)+v'_n+e_n-e_n  \}\\
=&v'+T_{v'}(u').
\end{align*}

\begin{proposition}
The Hopf algebra $S(\mathcal{H}_{in} \times \mathcal{B}_1)$ is in cointeraction with $(\mathcal{H},\odot,\Delta)$ by the coaction $\rho_S$.
\end{proposition}

\begin{proof}
All the arguments from the proof of Theorem \ref{Thm::Cointeraction} apply.
\end{proof}

Now extend the coaction to $\rho_S: \mathcal{H}_{in}\times \mathcal{B}_1 \to S(\mathcal{H}_{in}\times \mathcal{B}_1) \otimes S(\mathcal{H}_{in}\times \mathcal{B}_1)$ by letting it act on the first component and then to $S(\mathcal{H}_{in}\times \mathcal{B}_1)$ as an algebra morphism. Then:
\allowdisplaybreaks
\begin{align*}
\langle v \otimes u \otimes \chi,(Id \otimes \rho_S)\rho_S(x)\rangle=&\langle v \otimes S_u(\chi),\rho_S(x)\rangle\\
=&\langle S_v(S_u(\chi)),x\rangle\\
=&\langle S_{S_v(u)}(\chi),x\rangle \\
=& \langle T_v(u)\otimes \chi, \rho_S(x)\rangle \\
=&\langle v \otimes u \otimes \chi, (\rho_S \otimes Id)\rho_S(x)\rangle.
\end{align*}
Hence $\rho_S$ is a coassociative coproduct on $S(\mathcal{H}_{in}\times \mathcal{B}_1)$. This corresponds to condition \eqref{eq::CoTranslationAssociative} for translations.

\begin{proposition} 
Let $(\mathcal{H},\odot,\Delta)$ be a non-degenerate combinatorial Hopf algebra that is in cointeraction with $(S(\mathcal{H}_{in}\times \mathcal{B}_1),\centerdot,\Delta_{\centerdot})$ by a coaction $\rho_S: \mathcal{H} \to S(\mathcal{H}_{in}\times \mathcal{B}_1) \otimes \mathcal{H}$, satisfying
\allowdisplaybreaks\begin{align*}
(Id \otimes \rho_S)\rho_S=&(\rho_S \otimes Id)\rho_S, \\
\langle e^v \otimes e_i, \rho_S(x)\rangle =&\langle v_i,x\rangle.
\end{align*}
Then the dual map $S_v$ given by
\allowdisplaybreaks\begin{align*}
\langle S_v(y),x\rangle = \langle e^v \otimes y,\rho_S(x)\rangle
\end{align*}
is a substitution.
\end{proposition}
\begin{proof}
This is straightforward to check in the same way it was done for translations.
\end{proof}

We now relate the coactions $\rho_S$ and $\rho_T$. The following proposition states that if one knows $\rho_S$, one can obtain $\rho_T$ by replacing every occurrence of an element $(e_i,e_i)$ in the left tensor by $1+(e_i,e_i)$. This can be understood intuitively: an $e_i$ in a translation $T_{v}(e_i)=e_i+v_i$ can result either from the identity part of the translation or from the $v_i$, while an $e_i$ in a substitution $S_v(e_i)=v_i$ can only follow from the $v_i$.

\begin{proposition} \label{prop::TranslationSubstitutionDualRelation}
Let $S$ be a substitution and let $T$ be the translation induced by $T_{v'}=S_v$ for $v'=\{v_1-e_1,\dots,v_n-e_n\}$. Define the linear map $\phi: S(\mathcal{H}_{in}\times \mathcal{B}_1) \to S(\mathcal{H}_{in}\times \mathcal{B}_1)$ by
\allowdisplaybreaks 
\begin{align*}
	\phi((x,e_i))=	\begin{cases*}
				1+(e_i,e_i), \quad x=e_i \\
				1, \quad \text{otherwise}
				\end{cases*}
\end{align*}
and
\allowdisplaybreaks
\begin{align*}
	\phi(x \centerdot y)=\phi(x)\centerdot \phi(y).
\end{align*}
Then
\allowdisplaybreaks
\begin{align*}
\rho_T=(\phi \otimes Id)\rho_S.
\end{align*}
\end{proposition}

\begin{proof}
By the assumption $T_{v'}=S_v$, we get
\allowdisplaybreaks
\begin{align*}
	\langle e^{v'}\otimes \chi,\rho_T(x)\rangle = \langle e^v \otimes \chi, \rho_S(x)\rangle.
\end{align*}
Denote
\allowdisplaybreaks
\begin{align*}
\epsilon=\sum_i (e_i,e_i)
\end{align*}
and note that
\allowdisplaybreaks
\begin{align*}
	e^{v} =&e^{v'+\epsilon}\\
		=&e^{v'}\centerdot e^{\epsilon}.
\end{align*}
Hence
\allowdisplaybreaks
\begin{align*}
	\langle e^{v'}\otimes \chi, \rho_T(x)\rangle 
	=& \langle e^{v'}\otimes e^{\epsilon} \otimes \chi, (\Delta_{\centerdot} \otimes Id)\rho_S(x)\rangle
\end{align*}
and the proposition follows.
\end{proof}

\begin{example}
Let $\mathcal{H}_{BCK}$ be the Hopf algebra of undecorated rooted trees corresponding to branched rough paths and let $T_v$ be the translation of $\mathcal{H}_{BCK}$ described in \cite{BrunedChevyrevFrizPreiss2017}. Then the coaction $\rho_S$ for the corresponding substitution map $S_v$ agrees with $\Delta_{CEFM}$ when restricted to trees. Then $\bullet$ is the unique element in $\mathcal{B}_1$ and:
\begin{align*}
\rho_S(\Forest{[[[]][]]})=&\bullet \bullet \bullet \bullet \otimes \Forest{[[[]][]]} + \Forest{[[]]}\bullet \bullet \otimes (2\Forest{[[][]]} + \Forest{[[[]]]} )+(\Forest{[[[]]]}\bullet +\Forest{[[][]]}\bullet + \Forest{[[]]}\Forest{[[]]}) \otimes \Forest{[[]]} +\Forest{[[[]][]]}\otimes \bullet , \\
\rho_T(\Forest{[[[]][]]})=&(\phi \otimes Id)\rho_S(\Forest{[[[]][]]}) \\
=& (\phi \otimes Id)(\bullet \bullet \bullet \bullet \otimes \Forest{[[[]][]]} + \Forest{[[]]}\bullet \bullet \otimes (2\Forest{[[][]]} + \Forest{[[[]]]} )+(\Forest{[[[]]]}\bullet +\Forest{[[][]]}\bullet + \Forest{[[]]}\Forest{[[]]}) \otimes \Forest{[[]]} +\Forest{[[[]][]]}\otimes \bullet    ) \\
=&(1+\bullet)^4 \otimes \Forest{[[[]][]]} + \Forest{[[]]}(1+\bullet)^2 \otimes (2\Forest{[[][]]} + \Forest{[[[]]]} )+(\Forest{[[[]]]}(1+\bullet) +\Forest{[[][]]}(1+\bullet) + \Forest{[[]]}\Forest{[[]]}) \otimes \Forest{[[]]} +\Forest{[[[]][]]}\otimes \bullet \\
=&1 \otimes \Forest{[[[]][]]} + \bullet \otimes \Forest{[[[]][]]} + \bullet \bullet \otimes \Forest{[[[]][]]} + \bullet \bullet \bullet \otimes \Forest{[[[]][]]} + \bullet \bullet \bullet \bullet \otimes \Forest{[[[]][]]} \\\
+&\Forest{[[]]} \otimes (2\Forest{[[][]]} + \Forest{[[[]]]} ) + \Forest{[[]]}\bullet \otimes (2\Forest{[[][]]} + \Forest{[[[]]]} ) + \Forest{[[]]}\bullet \bullet \otimes (2\Forest{[[][]]} + \Forest{[[[]]]} ) \\
+&\Forest{[[[]]]} \otimes \Forest{[[]]} + \Forest{[[[]]]}\bullet \otimes \Forest{[[]]} + \Forest{[[][]]} \otimes \Forest{[[]]} + \Forest{[[][]]}\bullet \otimes \Forest{[[]]} + \Forest{[[]]}\Forest{[[]]} \otimes \Forest{[[]]} + \Forest{[[[]][]]}\otimes \bullet.
\end{align*}
We see that, to go from $\rho_S$ to $\rho_T$, we have to identify every occurence of $\bullet$ on the left side of the tensor. Then we split the terms with $\bullet$ into a sum of either keeping the $\bullet$ on the left side, or replacing it with the unit for the multiplication.
\end{example}

\section{Substitutions from products} 
\label{section::Algebraic}

In \cite{BrunedChevyrevFrizPreiss2017}, the authors construct translations on the Butcher--Connes--Kreimer Hopf algebra $\mathcal{H}_{BCK}$ by considering a pre-Lie product on the primitive elements of the dual algebra. Noting that all primitive elements could be freely generated from $\mathcal{B}_1$ by the pre-Lie product, they define
\begin{align*}
	T_v(e_i)=e_i+v_i
\end{align*}
and then extend this map to a pre-Lie algebra morphism, as well as a morphism for the convolution product. We would like to capture this idea in the notion of \textit{subtitutions from products}.\\

Let $\mathcal{H}$ be a combinatorial Hopf algebra and suppose that we want to define a substitution map $S_v: \overline{\mathcal{H}^{\ast}} \to \overline{\mathcal{H}^{\ast}}$. If $S_v$ is defined on the primitive elements, then the property of being a morphism for the convolution product will uniquely determine $S_v$ on the whole space. Furthermore, $S_v$ must be a morphism for the Lie bracket on the primitives given by anti-symmetrisation of the convolution product. The problem of defining a substitution map for a given Hopf algebra then reduces to, given the values $S_v(e_i)=v_i$, extending the map $S_v$ to all primitive elements such that the extension is a Lie morphism.\\

In the case of geometric rough paths, the primitive elements are exactly the Lie polynomials generated by $\mathcal{B}_1$. Hence the assumption of $S_v$ being a convolution morphism, and therefore a Lie morphism, uniquely gives an extension to all primitives.\\

For branched rough paths, one can see by counting dimensions that being a Lie morphism is not sufficient to generate all primitive elements. If there are $n$ colours in $\mathcal{H}^{\mathcal{C}}_{BCK}$ then there are $\frac{n(n-1)}{2}$ linearly independent ways to combine degree one elements into degree two elements using Lie brackets, which is less than the $n^2$ different trees of degree two. The pre-Lie product is a suitable choice to generate the remaining primitive elements because it can be obtained by projecting the convolution product onto the primitives. This does in particular mean that the Lie bracket obtained by antisymmetrisation of the pre-Lie product coincides with the Lie bracket from the convolution product, so that a pre-Lie morphism is automatically also a Lie morphism. As a non-example we could generate all primitive elements using the Butcher product, which is given by grafting on only the root. But being a morphism for the Butcher product contradicts being a morphism for the Lie bracket, and hence can't give a substitution map.

\begin{definition}
Let $S_v : \overline{\mathcal{H}^{\ast}} \to \overline{\mathcal{H}^{\ast}}$ be a substitution of $\mathcal{H}$-rough paths. Suppose that there are $k$ products $\diamond_i: (\mathcal{H}_{in})^{\ast} \otimes (\mathcal{H}_{in})^{\ast} \to (\mathcal{H}_{in})^{\ast}$, $i=1,\dots,k$, such that $(\mathcal{H}_{in})^{\ast}$ is generated by $(\mathcal{B}_1)^{\ast}$ via these products. If $S_v(x \diamond_i y)=S_v(x)\diamond_i S_v(y)$, we say that $S_v$ is a $\diamond_i$-substitution. The algebra $((\mathcal{H}_{in})^{\ast},[\cdot,\cdot]_{\ast},\diamond_1,\dots,\diamond_k  )$ is called internally free\footnote{The name internally free was proposed in \cite{Preiss2021}} if $S_v(e_i)=v_i$ extends in a well-defined way for every $v$.
\end{definition}

\begin{theorem} \label{Thm::AlgebraicSubstitution}
Let $(\mathcal{H},\odot,\Delta)$ be a non-degenerate combinatorial Hopf algebra with basis $\mathcal{B}$. Suppose that $((\mathcal{H}_{in})^{\ast},\diamond_1,\dots,\diamond_k,[\cdot,\cdot]_{\ast})$ is internally free, generated by $\mathcal{B}_1$. Then the continuous map $S_v: \overline{\mathcal{H}^{\ast}} \to \overline{\mathcal{H}^{\ast}}$ defined by
\allowdisplaybreaks
\begin{alignat*}{2}
	S_v(e_i)=&v_i, &\qquad e_i \in& (\mathcal{B}_1)^{\ast}, \\
	S_v(a \diamond_i b)=& S_v(a) \diamond_i S_v(b) & a,b \in& (\mathcal{H}_{in})^{\ast}, \\
	S_v(x \ast y)=& S_v(x)\ast S_v(y) & x,y \in & \mathcal{H}^{\ast},
\end{alignat*}
is a substitution (and hence a translation).
\end{theorem}

\begin{proof}
We check all of the conditions:
\begin{enumerate}
\item $S_v(e_i)=v_i$ is by definition verified.
\item $S_v \circ S_u$ and $S_{S_v(u)}$ are both $\diamond_i$ morphisms, for $i=1,\dots,k$, that agree on $(\mathcal{B}_1)^{\ast}$, hence they agree on $(\mathcal{H}_{in})^{\ast}$. They are furthermore $\ast$-morphisms that agree on $(\mathcal{H}_{in})^{\ast}$ and therefore agree on $\mathcal{H}^{\ast}$. Finally they agree on $\overline{\mathcal{H}^{\ast}}$ by continuity.
\item $S_v$ maps characters to characters as it is a continuous algebra morphism that maps primitive elements to primitive elements.
\item $S_v(x\ast y)=S_v(x)\ast S_v(y)$ is by definition.
\item The bound can be seen from that $|S_v(x)| \leq N |x|$, where $N=\max \{|v_1|,\dots,|v_n| \}$, and $x \in \mathcal{H}^{\ast}$.
\end{enumerate}
\end{proof}

We are now interested in describing the coaction $\rho_S$. It turns out that the coaction can always be described using coloured operads. The construction used here is based on Foissy \cite{Foissy2017}, where coproducts are deduced from operads. This was adapted in \cite{Rahm2022} to construct coactions.\\ 

Suppose that $((\mathcal{H}_{in})^{\ast},\diamond_1,\dots,\diamond_k,[\cdot,\cdot]_{\ast})$ is internally free, then one can construct a coloured operad $P=\oplus_{m=1}^{\infty} P(m)$. Every element in $(\mathcal{H}_{in})^{\ast}$ can be expressed as a polynomial in elements from $(\mathcal{B}_1)^{\ast}$ by using the products $[\cdot,\cdot]_{\ast},\diamond_1,\dots,\diamond_k$. An element in $P(m)$ is a pair $(x,e_i)$ where $e_i \in \mathcal{B}_1$ and $x$ is a homogeneous element of degree $m$ in $(\mathcal{H}_{in})^{\ast}$, together with a bijection between the set $\{1,\dots,m \}$ and the degree $1$ elements in its polynomial representation. Let $y \in P(n)$ and $x_1,\dots,x_n \in P$, then the composition
\begin{align*}
(x_1,\dots,x_n) \circ y
\end{align*}
is defined if the second component of each $x_i$ equals the degree $1$ element labeled by $i$ in $y$. If this is the case, the composition is given by replacing each degree $1$ element in the polynomial representation of $y$ by the first component of their corresponding $x_i$. The labels of the degree $1$ elements in each $x_i$ are shifted by $\sum_{j<i}|x_j|$, so that the result of the composition remains in $P$. This is well-defined because $((\mathcal{H}_{in})^{\ast},\diamond_1,\dots,\diamond_k,[\cdot,\cdot]_{\ast})$ was assumed to be internally free. \\

We can now construct a module $R=\oplus_{m=1}^{\infty}R(m)$ over the operad $P$. Every element in $(\mathcal{H})^{\ast}$ can be expressed as a polynomial in elements from $(\mathcal{B}_1)^{\ast}$ by using the products $\ast,\diamond_1,\dots,\diamond_k$. An element in $R(m)$ is a pair $(x,e_i)$ where $e_i \in \mathcal{B}_1$ and $x$ is a homogeneous element of degree $m$ in $(\mathcal{H})^{\ast}$, together with a bijection between the set $\{1,\dots,m \}$ and the degree $1$ elements in its polynomial representation. Let $y \in R(n)$ and $x_1,\dots, x_n \in P$, then the composition
\allowdisplaybreaks
\begin{align*}
	(x_1,\dots,x_n) \circ y
\end{align*}
is defined if the second component of each $x_i$ equals the degree $1$ element labeled by $i$ in $y$. If this is the case, the composition is given by replacing each degree $1$ element in the polynomial representation of $y$ by the first component of their corresponding $x_i$. The labels of the degree $1$ elements in each $x_i$ are shifted by $\sum_{j<i}|x_j|$, so that the result of the composition remains in $R$. This is well-defined because $((\mathcal{H}_{in})^{\ast},\diamond_1,\dots,\diamond_k,[\cdot,\cdot])$ was assumed to be internally free. \\

Let $\pi_R: R \to \mathcal{H}$ denote the map given by forgetting the labels and the second component. Similarly let $\pi_P : P \to \mathcal{H}_{in} \times \mathcal{B}_1$ be the map given by forgetting the labels. For $x \in R$, let $x_{e_i}$ denote the number of times $e_i\in \mathcal{B}_1$ appears as a factor in $x$ and let
\allowdisplaybreaks
\begin{align*}
	\lambda(x)=\prod_{e_i \in \mathcal{B}_1}x_{e_i}!.
\end{align*}
Then one can construct a coaction
\allowdisplaybreaks
\begin{align*}
	\rho_S : \mathcal{H} \to S(\mathcal{H}_{in}\times \mathcal{B}_1) \otimes \mathcal{H}
\end{align*}
by
\allowdisplaybreaks
\begin{align*}
	\rho_S(x)=\sum_{y_1,\dots,y_n,z} \frac{1}{\lambda(z)\cdot |y_1|! \cdot \ldots \cdot |y_n|!} 
	\langle\pi_R( (y_1,\dots,y_n)\circ z ),x\rangle \pi_P(y_1)\centerdot \dots \centerdot \pi_P(y_n) \otimes \pi_R(z).
\end{align*}
\begin{proposition}
Let $\rho_S$ be as above and let $S_v$ be the algebraic from Theorem \ref{Thm::AlgebraicSubstitution}. Then
\allowdisplaybreaks
\begin{align*}
	\langle S_v(x),y\rangle = \langle e^v \otimes x, \rho_S(y) \rangle,
\end{align*}
for all $y \in \mathcal{H}^{\ast}$ and $x \in \mathcal{H}$.
\end{proposition}

\begin{proof}
Write $x$ as a polynomial in elements from $\mathcal{B}_1$, then $S_v(x)$ is obtained by replacing each $e_i$ by $v_i$. This means that there exists some way to label the factors of $x$ such that
\allowdisplaybreaks
\begin{align*}
	S_v(x) =\pi_R(((v_1,e_1),\dots,(v_1,e_1),(v_2,e_2),\dots,(v_n,e_n),\dots(v_n,e_n))\circ x),
\end{align*}
where the number of $(v_i,e_i)$ occurring in the composition is $x_{e_i}$, for $i=1,\dots,n$. The number of possible ways to do this labelling of $x$ is $\lambda(x)$ and the number of possible ways to label each $v_i$ is $|v_i|!$. This proves the proposition.
\end{proof}

Note that once we have a description of $\rho_S$, we can find a description of $\rho_T$ by applying Proposition \ref{prop::TranslationSubstitutionDualRelation}.

\section{Post-Lie translations in planarly branched rough paths}
\label{section::Examples}

We construct translations in planarly branched rough paths based on section \ref{section::Algebraic}.\\

Let $\mathcal{H}_{MKW}^{\mathcal{C}}$ be the Munthe-Kaas--Wright Hopf algebra. The indecomposable elements in $\mathcal{H}_{MKW}^{\mathcal{C}}$ are the Lie polynomials of trees, meaning all ordered sequences of trees generated by the Lie bracket $[\tau_1,\tau_2]=\tau_1\tau_2-\tau_2\tau_1$ acting on trees and on brackets of trees. This describes the free Lie algebra generated by $\mathcal{PT}_{\mathcal{C}}$. Hence $(\mathcal{H}_{MKW}^{\mathcal{C}})_{in}$ can be endowed with the structure of a free post-Lie algebra $( (\mathcal{H}_{MKW}^{\mathcal{C}})_{in},\graft,[\cdot,\cdot]   )$, where $\graft$ is the planar grafting product. Furthermore endowing $(\mathcal{H}_{MKW}^{\mathcal{C}})_{in}$ with the Lie bracket
\begin{align*}
[\tau_1,\tau_2]_{\ast}=\tau_1 \ast \tau_2 - \tau_2 \ast \tau_1
\end{align*}
does not break internal freeness, as the relation
\begin{align*}
[\tau_1,\tau_2]_{\ast}=\tau_1 \graft \tau_2 - \tau_2 \graft \tau_1 + [\tau_1,\tau_2]
\end{align*}
applies to all $\tau_1,\tau_2 \in (\mathcal{H}_{MKW}^{\mathcal{C}})_{in}$. Hence the construction from Section \ref{section::Algebraic} applies, we define post-Lie translations for planarly branched rough paths as translations generated by the post-Lie products. Note that at least two products are required to define a translation, as the dimensions of homogeneous components of $(\mathcal{H}_{MKW}^{\mathcal{C}})_{in}$ grows too fast to be generated by a single product. Let $T_v$ denote the post-Lie translation in $\mathcal{H}_{MKW}^{\mathcal{C}}$. Then to compute $T_v(\omega)$ for some $\omega$, we need to factorize $\omega$ by $\graft,[\cdot,\cdot],\ast$, e.g.

\begin{align*}
	T_{\{\Forest{[1[2]]},\Forest{[1]}\}}(\Forest{[1[1][2]]}\Forest{[2]})
	=&T_{\{\Forest{[1[2]]},\Forest{[1]}\}}(\Forest{[1[1][2]]}\ast \Forest{[2]} - \Forest{[2[1[1][2]]]} )\\
	=&T_{\{\Forest{[1[2]]},\Forest{[1]}\}}((\bullet_2 \graft (\bullet_1 \graft \bullet_1) -(\bullet_2 \graft \bullet_1)\graft \bullet_1 )\ast \bullet_2 - (\bullet_2 \graft (\bullet_1 \graft \bullet_1) -(\bullet_2 \graft \bullet_1)\graft \bullet_1)\graft \bullet_2 )\\
	=&((\bullet_2 + \bullet_1)\graft ( (\bullet_1+\Forest{[1[2]]}) \graft (\bullet_1+\Forest{[1[2]]}) ) - ( (\bullet_2+\bullet_1)\graft (\bullet_1 + \Forest{[1[2]]}) )\graft (\bullet_1+\Forest{[1[2]]})) \ast (\bullet_2 + \bullet_1) \\
	-& ((\bullet_2 + \bullet_1)\graft ( (\bullet_1+\Forest{[1[2]]}) \graft (\bullet_1+\Forest{[1[2]]}) ) - ( (\bullet_2+\bullet_1)\graft (\bullet_1 + \Forest{[1[2]]}) )\graft (\bullet_1+\Forest{[1[2]]})) \graft (\bullet_2 + \bullet_1).
\end{align*}

Factorizing forests in terms of the Grossman--Larson product is not how we like to think about forests, we rather prefer to think of them as concatenation products of trees. Let $\omega_1 \cdot \omega_2$ denote the noncommutative associative concatenation of the forests $\omega_1,\omega_2$. We will show that post-Lie translations of planarly branched rough paths are also morphisms for the concatenation product, which simplifies computations. As an example, the computation above can be done by factorization in terms of concatenation and $\graft$: 

\begin{align*}
T_{\{\Forest{[1[2]]},\Forest{[1]}\}}(\Forest{[1[1][2]]}\Forest{[2]})
=&T_{\{\Forest{[1[2]]},\Forest{[1]}\}}(\Forest{[1[1][2]]}\cdot \Forest{[2]})\\
=&T_{\{\Forest{[1[2]]},\Forest{[1]}\}}( \bullet_2 \graft (\bullet_1 \graft \bullet_1) -(\bullet_2 \graft \bullet_1)\graft \bullet_1) \cdot T_{\{\Forest{[1[2]]},\Forest{[1]}\}}(\bullet_2) \\
=&( (\bullet_2+\bullet_1) \graft ((\bullet_1+\Forest{[1[2]]}) \graft (\bullet_1+\Forest{[1[2]]})) -((\bullet_2+\bullet_1) \graft (\bullet_1+\Forest{[1[2]]}))\graft (\bullet_1+\Forest{[1[2]]}) )\cdot (\bullet_2 + \bullet_1).
\end{align*}

We now give the proof that post-Lie translations of planarly branched rough paths are concatenation-morphisms.

\begin{lemma} \label{lemma::TranslationsCommuteWithMapsInSums}
Let $(\mathcal{H},\odot,\Delta)$ be a combinatorial Hopf algebra and let $T_v: \overline{\mathcal{H}^{\ast}} \mapsto \overline{\mathcal{H}^{\ast}}$ be a translation in $\mathcal{H}$. Then:
\begin{align*}
\sum_{x \in \mathcal{B}} \langle a,x\rangle T_v(\delta_x)=\sum_{x \in \mathcal{B}}\langle T_v(a),x\rangle \delta_x,
\end{align*}
for all $a \in \overline{\mathcal{H}^{\ast}}$, where $\delta_x \in \mathcal{H}^{\ast}$ is the basis element that is dual to $x$.
\end{lemma}

\begin{proof}
Seeing $a \in \overline{\mathcal{H}^{\ast}}$ as a (possibly infinite) sum in the dual basis, the left side of the equation is given by the sum of applying $T_v$ to each of the terms. This is however also what the right side of the equation describes.
\end{proof}

\begin{proposition}
Let $T_v$ be the post-Lie translation map for planarly branched rough paths. Then $T_v$ is a morphism for the noncommutative associative concatenation product.
\end{proposition}

\begin{proof}
Define the map $\deg: (\mathcal{H}_{MKW}^{\mathcal{C}})^{\ast}\to \mathbb{N}$ by
\allowdisplaybreaks
\begin{align*}
	\deg(\omega)=\inf \{n: \langle \omega,\tau_1\cdot\ldots\cdot \tau_N\rangle=0, \forall N>n, 
	\; \forall \tau_1,\dots,\tau_N \in \mathcal{T}_{\mathcal{C}} \}.
\end{align*}
Then we can see from Equation \eqref{eq::PlanarGL} that
\allowdisplaybreaks
\begin{align*}
	\omega_1 \ast \omega_2 
	= \omega_1 \cdot \omega_2 + \sum_{\deg(\omega_3)<\deg(\omega_1 \ast \omega_2)} \langle \omega_1 \ast \omega_2,\omega_3\rangle \omega_3,
\end{align*}
when $\omega_1,\omega_2$ are homogeneous with respect to the map $\deg$. Then by Lemma \ref{lemma::TranslationsCommuteWithMapsInSums}:
\allowdisplaybreaks
\begin{align*}
	T_v(\omega_1 \cdot \omega_2)
	=&T_v(\omega_1 \ast \omega_2 -\sum_{\deg(\omega_3)<\deg(\omega_1\ast \omega_2)}\langle \omega_1 \ast \omega_2, \omega_3 \rangle \omega_3  ) \\
	=&T_v(\omega_1) \ast T_v(\omega_2)-\sum_{\deg(\omega_3)<\deg(\omega_1\ast \omega_2)}\langle \omega_1 \ast \omega_2, \omega_3 \rangle T_v(\omega_3) \\
	=&T_v(\omega_1)\ast T_v(\omega_2) - \sum_{\deg(\omega_3)<\deg(T_v(\omega_1 \ast \omega_2))}\langle T_v(\omega_1 \ast \omega_2),\omega_3 \rangle \omega_3 \\
	=&T_v(\omega_1)\ast T_v(\omega_2) - \sum_{\deg(\omega_3)<\deg(T_v(\omega_1 \ast \omega_2))}\langle T_v(\omega_1) \ast T_v(\omega_2),\omega_3 \rangle \omega_3 \\
	=&T_v(\omega_1)\cdot T_v(\omega_2).
\end{align*}
Hence $T_v$ is a concatenation morphism for homogeneous elements. The proposition then follows from linearity and continuity.
\end{proof}

\begin{remark}
In the construction of translations in geometric rough paths, the product on the primitive elements can be obtained by projecting the convolution product. The convolution product in geometric rough paths is concatenation. Concatenating two primitive elements and then projecting the result onto the primitives, is the same as taking half the Lie bracket of the two primitive elements. Similarly in branched rough paths, one can obtain the pre-Lie product on the primitives by first applying the convolution product and then projecting the result onto the primitive elements. There is a corresponding construction for planarly branched rough paths. Consider the Munthe-Kaas--Wright Hopf algebra endowed with a second coproduct $(\mathcal{OF}_{\mathcal{C}},\shuffle,\Delta_{MKW},\Delta_{\cdot})$, the deconcatenation coproduct which is dual to concatenation. The dual Hopf algebra then has two convolution products, planar Grossman--Larson and concatenation. Applying the above construction of first taking the convolution product and then projecting onto the primitive elements, to both of these convolution products, gives us two products on the primitive elements. The product obtained from concatenation is half the Lie bracket, the product obtained from planar Grossman--Larson is the sum of post-Lie grafting and half the Lie bracket. Defining a translation map to be a morphism with respect to these two products is equivalent to a post-Lie translation. Furthermore, by the above theorem, these translations are morphisms for both convolution products.
\end{remark}

We are now interested in describing the coaction $\rho_S$ that is dual to post-Lie substitution. We can then find $\rho_T$ by using Proposition \ref{prop::TranslationSubstitutionDualRelation}. A description of $\rho_S$, in the case of uncoloured trees, was derived in \cite{Rahm2022} by using the construction described in section \ref{section::Algebraic}. Extending the description to trees with coloured vertices is trivial. \\

The coaction $\rho_S: \mathcal{H}_{MKW}^{\mathcal{C}} \to S((\mathcal{H}_{MKW}^{\mathcal{C}})_{in} \times \mathcal{B}_1 ) \otimes \mathcal{H}_{MKW}^{\mathcal{C}}$ is given by contractions of \textit{admissible subforests}.

\begin{definition}
Let $\omega$ be a forest and let $\omega_1 \cdots \omega_n$ be a partition of the vertices of $\omega$ into subforests. This partition is admissible if and only if the following conditions are met:
\begin{enumerate}
\item Each root in the same $\omega_i$ are either roots of $\omega$ or grafted onto the same vertex of $\omega$. Furthermore, the roots of $\omega_i$ are adjacent in the planar embedding of $\omega$.
\item If $e$ is an edge in an $\omega_i$, then every edge $e'$ in $\omega$ that is outgoing from the same vertex as $e$ and is to the right of $e$ in the planar embedding, is also in $\omega_i$.
\end{enumerate}
If $\omega_1 \cdots \omega_n$ is an admissible subforest of $\omega$, let the contraction $\omega / \omega_1 \cdots \omega_n$ deonte the sum of all forests obtained by contracting each $\omega_i$ into a single vertex.
\end{definition}

The cosubstitution coaction $\rho_S$ is now given as a sum over all admissible subforests, tensored with all the corresponding contractions. If an $\omega_i$ has several roots, then Lie brackets has to be inserted in the left tensor.

\begin{example}Let $\mathcal{C}=\{1,\dots,k\}$, then:
\begin{align*}
\rho_S(\Forest{[1[3][2]]})=&\sum_{i=1}^k (\Forest{[1[3][2]]},\Forest{[i]}) \otimes \Forest{[i]}+\sum_{i,j=1}^k ([\Forest{[2]},\Forest{[3]}],\Forest{[i]})\centerdot (\Forest{[1]},\Forest{[j]})\otimes \Forest{[j[i]]}\\
+&\sum_{i,j=1}^k(\Forest{[2]},\Forest{[i]} )\centerdot (\Forest{[1[3]]},\Forest{[j]})\otimes \Forest{[j[i]]}+\sum_{i,j,\ell=1}^k(\Forest{[1]},\Forest{[i]})\centerdot (\Forest{[2]},\Forest{[j]})\centerdot (\Forest{[3]},\Forest{[k]})\otimes \Forest{[i[k][j]]},
\end{align*}
corresponding to the admissible partitions:
\begin{align*}
(\Forest{[1[3][2]]}),(\Forest{[2]}\Forest{[3]},\Forest{[1]}),(\Forest{[2]},\Forest{[1[3]]}),(\Forest{[2]},\Forest{[3]},\Forest{[1]}).
\end{align*}
\end{example}

We conclude by remarking on how post-Lie translations interact with differential equations driven by planarly branched rough paths.

\begin{proposition}
$Y_{st}$ is a solution to the controlled differential equation
\begin{align*}
	dY_{st}= \#f(Y_{st})d(T_v(\mathbb{X}))
\end{align*}
if and only if it is a solution to the controlled differential equation
\begin{align*}
	dY_{st}=\# \mathcal{F}_f(\{\bullet_1+v_1,\dots,\bullet_d+v_d \})(Y_{st})d\mathbb{X}.
\end{align*}
\end{proposition}

\begin{proof}
This is an immediate consequence from Lemma \ref{lemma::TranslationsCommuteWithMapsInSums} and definition \ref{def::PlanarRoughSolution}.
\end{proof}

\bibliographystyle{acm}
\bibliography{RoughReferences}
\end{document}